\documentclass[a4paper,10pt]{article}

\usepackage[english]{babel}
\usepackage[T1]{fontenc}
\usepackage{enumerate}
\usepackage{amsmath}
\usepackage{amssymb}
\usepackage{amsthm}

\newcommand{\Z}{\mathbb{Z}}
\newcommand{\N}{\mathbb{N}}

\newcommand{\rightleftarrow}{\leftrightarrow}
\newtheorem{theorem}{Theorem}[section]
\newtheorem{prop}[theorem]{Proposition}
\newtheorem{lemma}[theorem]{Lemma}

\newtheorem{definition}[theorem]{Definition}

\newtheorem{claim}[theorem]{Claim}

\title{Expansion of Presburger arithmetic with the Exchange Property}
\author{Nathana\"el Mariaule\footnote{During the preparation of this paper the author was supported by the Fonds de la Recherche Scientifique - FNRS}} 
\date{}
\begin{document}
\maketitle
\begin{abstract}
Let $G$ be a model of Presburger arithmetic. Let $\mathcal{L}$ be an expansion of the language of Presburger $\mathcal{L}_{Pres}$. In this paper we prove that the $\mathcal{L}$-theory of $G$ is $\mathcal{L}_{Pres}$-minimal iff it has the exchange property and any bounded definable set has a maximum.
\end{abstract}
\section{Introduction}
Let $G$ be a model of Presburger arithmetic i.e., $G$ is a group elementary equivalent to $\Z$ (as ordered group i.e. in the language $\mathcal{L}_{Pres}=(+,-,0,1,<,\equiv_n (n\in \N))$ - where $\equiv_n$ is the congruence relation).
We fix $\mathcal{L}$ an expansion of $\mathcal{L}_{Pres}$ and an interpretation of the elements of $\mathcal{L}$ in $G$. We will consider the theory of $G$ as a $\mathcal{L}$-structure. 
\begin{definition} We say that $(G,\mathcal{L})$ is $\mathcal{L}_{Pres}$-minimal if for all $X\subset G$ definable, $X$ is $\mathcal{L}_{Pres}$-definable.  $Th(G,\mathcal{L})$ is $\mathcal{L}_{Pres}$-minimal if for all $\mathcal{M}\equiv G$ (as $\mathcal{L}$-structures), $(\mathcal{M},\mathcal{L})$ is $\mathcal{L}_{Pres}$-minimal.
\end{definition}
\begin{Remark}
All definable sets are definable with parameters.
\end{Remark}
\par It is known that $\mathcal{L}_{Pres}$-minimal theories have many good properties. For instance in \cite{Cluckers}, R. Cluckers proves a cell decomposition theorem for models of such theory. It follows that it admits a good notion of dimension and the exchange property.
\begin{definition} Let $\mathcal{M}=(M,\cdots)$ be a $\mathcal{L}$-structure. Let $A\subset M$. Then $dcl(A)$ is the set of $a\in M$ such that there is $\Phi(x,\overline{y})$ a $\mathcal{L}$-formula and $\overline{b}\subset A$ such that $\mathcal{M}\vDash \Phi(a,\overline{b})\wedge \forall y \Phi(y,\overline{b})\rightarrow y=a$. 
\end{definition}
\begin{definition} A theory $T$ satisfies EP (exchange property) if for all $\mathcal{M}\vDash T$, for all $A\subset M$, for all $a,b\in M$, 
$$a\in dcl(A\cup \{b\})\setminus dcl(A)\mbox{ implies } b\in dcl(A\cup\{a\})$$
\end{definition}
\par If $G=\Z$, then the theory $Th(G,\mathcal{L})$ satisfies EP if and only if it is $\mathcal{L}_{Pres}$-minimal. This follows easily from results of C. Michaux and R. Villemaire \cite{M-V}. In their paper the authors prove that  (1) $Th(\Z,\mathcal{L})$ is $\mathcal{L}_{Pres}$-minimal iff $(\Z,\mathcal{L})$ is $\mathcal{L}_{Pres}$-minimal and (2) that $(\Z,\mathcal{L})$ is $\mathcal{L}_{Pres}$-minimal iff there is no definable expanding set in $\Z$.
\begin{definition} Let $X\subset G$ definable. We say that $X$ is expanding if for all $n\in \N$ there are infinitely many $x\in X$ such that $x+1,\cdots x+n\notin X$.
\end{definition}
It is then not difficult to prove that the exchange property implies that there is no definable expanding set:

\begin{prop}\label{EP implies no expanding set} Let $G\equiv \Z$. If $T=Th(G,\mathcal{L})$ satisfies EP then for all $\mathcal{M}\vDash T$, for all $X\subset M$ definable, $X$ is not expanding.
\end{prop}
\begin{proof}
Let $M$ be a model of $T$ and $X\subset M$ definable. Assume that $X$ is expanding. Then by compactness there is $M^*$ elementary expansion of $M$ and $x<y'\in X(M^*)$ such that $x\notin dcl(\varnothing)$, $x+n\notin X(M^*)$ for all $n\in \N$. Let $y$ be the successor of $x$ in $X(M^*)$. By compactness again there is an elementary extension $M'$ of $M^*$ and $z\in M'$ such that $x<z<y$ and $z\notin dcl(\{x\})$. Then, $x\in dcl(\{z\})$ as it is the only element in $M'$ that satisfies the formula $\Phi(t,z)\equiv t\in X\wedge t<z$.  
\par So $x\in dcl(\{z\})\setminus dcl(\varnothing)$ while $z\notin dcl(\{x\})$. This contradicts the assumption that $T$ has the exchange property.
\end{proof}

In general it is not true that EP implies $\mathcal{L}_{Pres}$-minimality. For instance consider $G$ be a nonstandard model of Presburger and the expansion of $\mathcal{L}_{Pres}$ by a unary predicate interpreted in $G$ by a proper convex subgroup. In fact, the existence a proper definable subset closed under successor or predecessor is the only remaining obstruction to $\mathcal{L}_{Pres}$-minimality.

\begin{definition} $(G,\mathcal{L})$ satisfies DC (definable completeness) if any definable unary set has a supremum in $G\cup\{\infty\}$.
\end{definition}
\begin{Remark} DC is a first-order property. So it is also a property of the theory of $G$.
\end{Remark}

For the rest of this paper, we assume that $G\vDash DC$ and that $Th(G,\mathcal{L})$ satisfies EP (so by Proposition \ref{EP implies no expanding set}, there is no definable expanding set). We fix $X$ a $\mathcal{L}$-definable subset of $G$. Under these hypotheses, we shall prove that $X$ is $\mathcal{L}_{Pres}$-definable (Proposition \ref{X Lpres-def}). Then, the main theorem follows:
\begin{theorem} $Th(G,\mathcal{L})$ is $\mathcal{L}_{Pres}$-minimal iff $Th(G,\mathcal{L})$ satisfies EP and DC.
\end{theorem}
\begin{proof}
One direction is done in \cite{Cluckers}:  EP follows for instance from cell-decomposition and DC is an immediate consequence of $\mathcal{L}_{Pres}$-minimality. 
The other direction will be proved in Proposition \ref{X Lpres-def}.
\end{proof}
The proof of Proposition \ref{X Lpres-def} has two main steps. First, we prove a non-standard version of Michaux-Villemaire \cite{M-V}. More precisely following their strategy we prove that if $X\subset G$ is definable then $X\cap\N$ is a finite union of points and of cosets of $dG$ (for some integer $d$). Then by DC, we can extend this property to an infinite interval $[0,g]$.
Finally by compactness and EP, we prove that $X$ is a finite union of $X_i$ were $X_i$ is an interval intersected with finitely many cosets of $d_iG$.
\begin{Remark} It is already known that the above theorem fails for generalisation of $\mathcal{L}_{Pres}$-minimality. For coset-minimal groups (in the sense of \cite{P-W}; note that $\mathcal{L}_{Pres}$-minimal groups are coset-minimal), there is an example in \cite{B-P-W} of coset-minimal group which does not have the exchange property.
\end{Remark}

\section{Proof of the main theorem.}

If $X$ is a subset of $M$ where $M\equiv G$ then for all $x\in X$, we denote the successor of $x$ in $X$ by $\sigma(x)$.
\begin{lemma}\label{non expanding set} Let $X$ be a non expanding set. Then there is $u\in \N$ such that for all $x\in X\cap\N$, if $\sigma(x)\in X\cap \N$. Then, $\sigma(x)-x\leq u$.
\end{lemma}
\begin{proof}
For we may assume that $X\cap\N$ is infinite (Otherwise the Lemma is trivial). By contradiction assume that for all $u\in \N$, there is $x\in X\cap\N$ such that $\sigma(x)-x>u$. Apply this assumption with $u=n$ and we get $x_0\in X\cap\N$ such that $\sigma(x_0)-x_0>n$. By induction we can construct $x_0<x_1<\cdots<x_{k}\in X\cap\N$ such that for all $i$, $\sigma(x_i)-x_i>n$. Indeed, we can apply the assumption with $u=\sigma(x_{k-1})+n$ (note that $u\in \N$ as $\sigma(x_{k-1})\in \N$ because $X\cap \N$ is infinite). Then we obtain $x_k$ with $\sigma(x_k)-x_k>u=\sigma(x_{k-1})+n$. In particular $\sigma(x_k)-x_k>n$ and $\sigma(x_k)>\sigma(x_{k-1})$ ie $x_k>x_{k-1}$. Let $U_n=\{x_k;\ k\in \N\}$. Then $U_n$ is an infinite subset of $X$. Therefore $X$ is expanding as for all $x\in U_n$, $x+1,\cdots, x+n\notin X$. Contradiction.
\end{proof}

 Set $G^+=\{x\in G:\ x\geq 0\}$ and $G^{>0}=G^+\setminus \{0\}$. For the rest of this paper, we will assume without loss of generality that $X\subset G^+$. 
Let $a\leq b\in G^+$. Let $X[a,b]=\{x\in G^+:\ a+x\in X,\ a+x\leq b\}$.
\begin{lemma}\label{explicit X[a,b]} Let $a\leq b\in G^+$. Let $g\in G^+$. Then $X[a,b]=X[a+g,b+g]$ if and only if for all $a\leq x\leq b$, $x\in X$ iff $x+g\in X$.
\end{lemma}
\begin{proof}
Immediate from the definition.
\end{proof}

\begin{lemma}\label{increase}
For all $a\leq c\leq d\leq b\in G^+$ for all $g\in G^+$, if $X[a,b]=X[a+g,b+g]$ then $X[c,d]=X[c+g,d+g]$.
\end{lemma}
\begin{proof}
By Lemma \ref{explicit X[a,b]}, for all $a\leq x\leq b$, $x\in X$ iff $x+g\in X$. In particular, this is the case for all $c\leq x\leq d$. By Lemma \ref{explicit X[a,b]} again, $X[c,d]=X[c+g,d+g]$.
\end{proof}

\begin{definition}
$$
\begin{array}{rl}
\widetilde{d}:& G^+\rightarrow G^+\cup \{-1\}: \\
& n\longmapsto
\left\{\begin{array}{ll}
\min\{g>0:\ \exists a\in G^+ X[a,a+n]=X[a+g,a+g+n]\} & \mbox{if such $g$ exists;}\\
-1 & \mbox{ otherwise.}
\end{array}\right.
\end{array}
$$
\end{definition}

\begin{lemma} For all $n'<n\in G^+$, if $\widetilde{d}(n)\not=-1$ then $0< \widetilde{d}(n')\leq \widetilde{d}(n)$.
\end{lemma}
\begin{proof}
Let $n'\leq n\in G^+$ with $\widetilde{d}(n)>0$. Then by definition of $\widetilde{d}(n)$ there is  $a\in G^+$ such that $X[a,a+n]=X[a+\widetilde{d}(n),a+n+\widetilde{d}(n)]$.
So by Lemma \ref{increase} $X[a,a+n']=X[a+\widetilde{d}(n),a+n'+\widetilde{d}(n)]$. Therefore by definition of $\widetilde{d}$, $0<\widetilde{d}(n')\leq \widetilde{d}(n)$.
\end{proof}

\begin{lemma}\label{function d} $\widetilde{d}(\N)\subset \N$.
\end{lemma}
\begin{proof}
First we remark that for all $k,n\in \N$, $X[k,k+n]\subset \{0,\cdots n\}$. Therefore by the Pigeonhole principle for all $n\in \N$, there is $k<l\in \N$ such that $X[k,k+n]=X[l,l+n]$. Now by definition of $\widetilde{d}$, $0<\widetilde{d}(n)\leq l-k\in \N$.
\end{proof}

\begin{definition}
$$
\begin{array}{rl}
\widetilde{a}:&G^+\times G^{>0} \rightarrow G^+\cup \{-1\}:\\
& (n,d)\longmapsto
\left\{\begin{array}{ll}
\min\{a\in G^+:\  X[a,a+n]=X[a+d,a+n+d]\} & \mbox{if such $a$ exists;}\\
-1 & \mbox{ otherwise.}
\end{array}\right.
\end{array}
$$
\end{definition}

\begin{lemma}\label{function a} For all $n\in \N$, there is $d\in \N_0$ such that $\widetilde{a}(n,d)\in \N$.
\end{lemma}
\begin{proof}
By the Pigeonhole principle for all $n\in \N$, there is $k<l\in \N$ such that $X[k,k+n]=X[l,l+n]$. Then by definition of $\widetilde{a}$, $0\leq \widetilde{a}(n,l-k)\leq k\in \N$. Take $d=l-k$.
\end{proof}

The functions $\widetilde{a},\widetilde{d}$ come from Michaux-Villemaire \cite{M-V}. The authors prove that if $G=\Z$, $\alpha(n):=\widetilde{a}(n,\widetilde{d}(n))\in \N$. In fact, it is proved that for $n$ large enough, $X\cap [\alpha(n),\infty)$ is defined by congruences relations modulo $d(n)$. In our case this is not true anymore as $\alpha(n)$ may not be in $\N$. For instance, take $X= (3G\cap [0, g]) \cup (2G\cap [g+1,\infty)$ for some $g\in G\setminus \N$. Then for $n\in \N$, $\widetilde{d}(n)=2$ and $\alpha(n)=g+1$. What we would like to capture is $\widetilde{d}(n)=3$ and $\alpha(n)=0$. For we defined the below function $D$ which is a twisted version of $\widetilde{d}$. With this new function we will get that $D(n)=2$ and $A(n):=\widetilde{a}(n,D(k))=0$ for all $k,n\in \N$ large enough as required. 

\begin{definition}
$$
\begin{array}{rl}
D:&G^+ \rightarrow G^+\cup \{-1\}:\\
& n \longmapsto
\left\{\begin{array}{ll}
\min\{d\in G^{>0}:\ \widetilde{a}(n,d)\not=-1\wedge\\
 \forall d'>d\ \widetilde{a}(n,d')\not=-1\rightarrow [\widetilde{a}(n,d)\leq \widetilde{a}(n,d') \vee \widetilde{a}(n,d')+d'>\widetilde{a}(n,d)+n]\}&\mbox{ if such $d$ exists}\\
-1 & \mbox{ otherwise.}
\end{array}\right.
\end{array}
$$
Let $A(n):=\widetilde{a}(n,D(n))$.
\end{definition}

\begin{lemma}\label{function D A} $D(\N)\subset \N$ and $A(\N)\subset \N$.
\end{lemma}
\begin{proof}
By Lemma \ref{function a}, for all $n\in \N$ there is $d\in \N_0$ such that $\widetilde{a}(n,d)\in \N$. Let $B_n:=\{d'\in \N_0:\ \widetilde{a}(n,d')\in \N\mbox{ and } \forall d''\in \N_0, \widetilde{a}(n,d'')\in \N\rightarrow \widetilde{a}(n,d')\leq \widetilde{a}(n,d'')\}$. This set is nonempty. For there is $d'$ such that $\widetilde{a}(n,d')$ is minimal in $\widetilde{a}(n,\N) \setminus\{-1\} $. Then as $\widetilde{a}(n,d)\in \N$,  $\widetilde{a}(n,d')\in \N$. By definition, $d'\in B_n$. 
\par Let $d^*=\min\{d'\in B_n\}\in \N_0$. Set 
$$E_n:=\{d\in G^{>0}: \widetilde{a}(n,d)\not=-1\wedge \forall d'>d\ \widetilde{a}(n,d')\not=-1\rightarrow [\widetilde{a}(n,d)\leq \widetilde{a}(n,d') \vee \widetilde{a}(n,d')+d'>\widetilde{a}(n,d)+n]\}.$$
Let $d'>d^*$ such that $\widetilde{a}(n,d')\not=-1$. First if $d'\in \N$ then $\widetilde{a}(n,d^*)\leq \widetilde{a}(n,d')$ (for either $\widetilde{a}(n,d')\notin \N$ and it is trivial as $\widetilde{a}(n,d^*)\in \N$ or $\widetilde{a}(n,d')\in \N$ and this follows from the definition of $B_n$). Second if  $d'\notin\N$ then $\widetilde{a}(n,d')+d'>k$ for all $k\in \N$. So $\widetilde{a}(n,d^*)+n<\widetilde{a}(n,d')+d'$. This proves that $d^*\in E_n$.
This implies that $D(n)\leq d^*$ by definition of $D$. So $D(\N)\in \N$.
\par Let $d\in \N$ such that $d=D(n)$. Assume that $A(n)=\widetilde{a}(n,d)\notin \N$. As $\widetilde{a}(n,d^*)\in \N$, $\widetilde{a}(n,d)>\widetilde{a}(n,d^*)$. As $D(n)\leq d^*$ we get that $d<d^*$. On the other hand, as $\widetilde{a}(n,d^*)+d^*\in \N$ and $\widetilde{a}(n,d)\notin \N$, we have that $\widetilde{a}(n,d)+n>\widetilde{a}(n,d^*)+d^*$. Now by definition of $D(n)$, $d\not=D(n)$. We get a contradiction. Therefore $A(n)\in \N$.
\end{proof}

\begin{definition} Let $U\subset V\subset G$. We say that $U$ is cofinal in $V$ if for all $v\in V$ there is $u\in U$ such that $v\leq u$. 
\end{definition}

\begin{lemma}\label{f nondecreasing}
Let $f:U\rightarrow G^+$ definable such that $U_0:=U\cap \N$ is cofinal in $\N$ and $f(U_0)\subset \N$ then there is $U'\subset U$ definable such that $U'\cap \N$ is cofinal in $\N$ and $f$ is nondecreasing on $U'$.
\end{lemma}
\begin{proof}
\par First, if $f(U_0)$ is finite: Then by the Pigeonhole principle there is $s\in Im\ f\cap U_0$ such that $f^{-1}(s)\cap U_0\subset G^+$ is infinite. Set $U'= f^{-1}(s)\cap U$. By definition of $U'$, $f$ is constant on it. Also $U'\cap \N$ is cofinal in $\N$ as it is infinite. Clearly $U'$ is definable.
\par Second, if $f(U_0)$ is infinite. In that case, set $U':=\{x\in U: \forall y\in U y<x\rightarrow f(y)< f(x)\}$. By definition $f$ is nondecreasing of $U'$ and $U'$ is definable. It remains to prove that $U_0':= U'\cap \N$ is an infinite set. We remark that $U_0'$ is non empty as $\min\{n:\ n\in U_0\}\in U_0'$. Let $y\in U_0'$. Assume that for all $t\in U_0$ with $t>y$, $f(t)\leq f(y)$. So, $f(U_0)= [0,f(y)]\subset\N$ is a finite set. This contradicts the assumption that $f(U_0)$ is an infinite subset of $\N$. So, there is $t>y$, $t\in U_0$ such that $f(y)<f(t)$. Let $t^*=\min\{t>y:\ t\in U,\ f(y)<f(t)\}$. As $t^*\leq t$, $t^*\in \N$. Let $z<t^*$. If $y\leq z<t^*$ then $f(z)\leq f(y)$ and $f(y)<f(t^*)$ by definition of $t^*$. If $z<y$ then $f(z)< f(y)$ (by definition of $U'$) and $f(y)<f(t)$ (by definition of $t^*$). So, by definition of $U'$, $t^*\in U'$. Therefore $t^*\in U_0'$. This proves that $U_0'$ is infinite and concludes the proof of the lemma.
\end{proof}

\begin{lemma}\label{D A nondecreasing}
There is $\widetilde{U}\subset G^+$ definable such that $\widetilde{U}\cap \N$ is cofinal in $\N$ and $D,A$ are nondecreasing on $\widetilde{U}$.
\end{lemma}
\begin{proof}
By Lemma \ref{function D A} and Lemma \ref{f nondecreasing}.
\end{proof}

\begin{lemma}\label{DA definable} $D,A: \widetilde{U}\rightarrow G^+$ are definable maps.
\end{lemma}
\begin{proof}
By definition, $D,A: G^+\rightarrow G^+$ are definable maps. By Lemma \ref{D A nondecreasing}, $\widetilde{U}$ is definable.
\end{proof}

\begin{lemma}\label{d constant}
Assume that $\widetilde{U}$, $Im\ D$ and $Im\ A$ are nonexpanding. Then, there is $d\in \N_0$, such that for all $n\in \N$, there is $a\in \N$ such that $X[a,a+n]=X[a+d,a+d+n]$.
\end{lemma}

\begin{proof}
Let us remark that it  is sufficient to prove that there is $d\in \N_0$, such that for all $n\in \N$, there is $a,m\in \N$ such that $m\geq n$ and $X[a,a+m]=X[a+d,a+d+m]$.	For in that case, by Lemma \ref{increase}, $X[a,a+n]=X[a+d,a+d+n]$. 
\par By Lemma \ref{non expanding set} and $D$ is nondecreasing, there is $u\in \N$ such that for all $n\in \widetilde{U}_0:=\widetilde{U}\cap \N$, $0\leq D(\sigma(n))-D(n)\leq u$. Also there is $v\in \N$ such that for all $n\in \widetilde{U}_0$, $0\leq A(\sigma(n))-A(n)\leq v$.
\begin{claim}
There are $V,N\in \N$ such that  for all  $n\in \N$,  $n\geq N$, there are $d(n)\in \{0,\cdots, V\}$ and  $a\in \N$ with  $X[a,a+n]=X[a+d(n),a+d(n)+n]$.
\end{claim}
\begin{proof}
	First if $(D(n), n\in\widetilde{U}_0)$ is eventually constant, take $V'=\lim D(n)$, $N=\min\{n: \forall m\geq n D(n)=D(m)\}$, $d(n)=V'$
		and $a=A(n,D(n))$.	By definitions of $A,D,d$, we are done.
	\par Otherwise, $\forall n'\in \widetilde{U}_0\ \exists k'\in\widetilde{U}_0,\ n'<k'$ and $D(n')<D(k')$ (as $D$ is nondecreasing).
	Apply this with $n'=n+v$. Then for all $n$ there is $k'$ such that $D(k')>D(n+v)$. Take $k=\max\{k^*\in\widetilde{U}_0: D(k^*)=D(n+v)\}$. Then $D(\sigma(k))>D(n+v)=D(k)$. We remark that by definition of $A,D$
	$$ (1) X[A(k), A(k)+k]=X[A(k)+D(k),A(k)+D(k)+k] \mbox{ and, }$$ 
	$$(2) X[A(\sigma(k)), A(\sigma(k))+\sigma(k)]=X[A(\sigma(k))+D(\sigma(k)),A(\sigma(k))+D(\sigma(k))+\sigma(k)].$$

\par We have that $A(k)+k-A(\sigma(k))=k-(A(\sigma)(k)-A(k))\geq n+v-v=n$. So,
 $$A(k)\leq A(\sigma(k))\leq A(k)+k\qquad \mbox{ and }\qquad A(\sigma(k))\leq A(k)+k\leq A(\sigma(k))+\sigma(k).$$
 Therefore by Lemma \ref{increase}, (1) and (2)
	$$ (1') X[A(\sigma(k)), A(k)+k]=X[A(\sigma(k))+D(k),A(k)+D(k)+k] \mbox{, and }$$ 
	$$(2') X[A(\sigma(k)), A(k)+k]=X[A(\sigma(k))+D(\sigma(k)),A(k)+D(\sigma(k))+k].$$
We combine $(1')$ and $(2')$ to get
$$X[A(\sigma(k))+D(k),A(k)+D(k)+k]=X[A(\sigma(k))+D(\sigma(k)),A(k)+D(\sigma(k))+k].$$
Take $a=A(\sigma(k))+D(k)$ and $d=d(k)= D(\sigma(k))-D(k)$. We remark that $a\in \N$, that $d\leq u$ (by definition of $u$) and that $X[a,a+n]=X[a+d,a+d+n]$ (for remark that $a\leq a+n\leq A(k)+D(k)+k$ and apply Lemma \ref{increase} and the above equality). Set $V=u$ and $N=0$.
\end{proof}
By the claim and the Pigeonhole principle there is $d\in \{1,\cdots, V\}$ such that $E=\{n:\ d(n)=d\}$ is cofinal in $\N$. Then, for all $n\in \N$ there is $m\in E$ such that $m\geq n$. So, by the above claim, there is $a$ such that $X[a,a+m]=X[a,a+d+m]$. This shows that there is $d\in \{1,\cdots, V\}$ such that for all $n$ there is $a,m\in \N$ such that $m\geq n$ and $X[a,a+m]=X[a+d,a+d+m]$. By the remark at the beginning of the proof we are done.
\end{proof}

\begin{definition}\label{alpha}Let $d$ given by Lemma \ref{d constant}. We define
$$
\begin{array}{rl}
\alpha:& \widetilde{U}\rightarrow G^+\cup \{-1\}: \\
& n\longmapsto
\left\{\begin{array}{ll}
\min\{a\geq 0:\ X[a,a+n]=X[a+d,a+n+d]\} & \mbox{if there is at least one such $a$;}\\
-1 & \mbox{ otherwise.}
\end{array}\right.
\end{array}
$$
\end{definition}

By Lemma \ref{d constant}, for all $n\in \widetilde{U}\cap \N$, $\alpha(n)\in \N$. Note that $\alpha$ is definable.

\begin{lemma}\label{alpha nondecreasing} There is $U\subset \widetilde{U}$ definable such that $U_0:= U\cap\N$ is cofinal in $\N$ and $\alpha$ is non decreasing on $U$.
\end{lemma}
\begin{proof}
By Lemma \ref{f nondecreasing}.
\end{proof}

From now, we will assume that $\alpha$ is restricted to $U$. We set $U_0:=U\cap \N$.

\begin{lemma}\label{loc definable} If $Im\ \alpha$, $\widetilde{U}$, $Im\ D$ and $Im\ A$ are nonexpanding, then there is $g\in G^+\cup\{\infty\}\setminus \N$ and $N\in \N$ such that $X\cap [N,g]$  is a finite union of coset of $dG$ intersected with $[N,g]$ (so is $\mathcal{L}_{Pres}$-definable).
\end{lemma}
\begin{proof}
By Lemma \ref{non expanding set}, there is $v$ such that for all $n$, $0\leq \alpha(\sigma(n))-\alpha(n)\leq v$. So for all $n\geq l:=\max\{d,v\}$, $\alpha(\sigma(n))\leq \alpha(n)+v\leq \alpha(n)+n$.
Set $N= \alpha(l)$. 
\par We have that for all $k\in \N$, $k\geq N$, there is $n\in U_0$, $n\geq l$ such that $k\in [\alpha(n),\alpha(n)+n]$. For there are two cases: first if there is $n\in U_0$ such that $\alpha(n)\leq k\leq\alpha(\sigma(n))$. In that case $k\in [\alpha(n), \alpha(\sigma(n))]\subseteq [\alpha(n),\alpha(n)+n]$ and we are done. Second if for all $n\in U_0$ $\alpha(n)<k$. In that case by Lemma \ref{alpha nondecreasing}, $\alpha$ is non decreasing. Furthermore, $\alpha(U_0)\subset \N$ (by Lemma \ref{d constant}). So as it is bounded by $k$, $\alpha$ is eventually constant in $U_0$ i.e., there is $M\in \N$ such that for all $n\in U_0$, $n>M$, $\alpha(n)=\alpha(\sigma(n))$. Let $n\in U_0$ such that $n>\max\{M, k\}$ (such $n$ exists as $U_0$ is cofinal in $\N$ see Lemma \ref{alpha nondecreasing}). Then, $\alpha(n)=\alpha(M)<k<\alpha(M)+n=\alpha(n)+n$.
\par Let $x\in [N,\infty)\cap \N$. By the above argument we know that there is $n$ such that $x\in [\alpha(n),\alpha(n)+n]$. So by definition of $\alpha$, $X[\alpha(n),\alpha(n)+n]=X[\alpha(n)+d,\alpha(n)+n+d]$. By Lemma \ref{explicit X[a,b]}, this implies that $x\in X$ iff $x+d\in X$. Therefore $X\cap [N,\infty)\cap \N= (X\cap [\alpha(l),\alpha(l)+d-1])+d\N$. Take $\{a_1,\cdots, a_k\}$ be the set $X\cap [\alpha(l),\alpha(l)+d-1]$. Then $X\cap [N,\infty)\cap \N=\bigcup_i (a_i+dG)\cap [N,\infty)\cap \N$.
\par Let $Y=\{x\in G^+: x\geq N \wedge \forall y\leq x,\ N\leq y\rightarrow  (y\in X\rightleftarrow \vee_i y\equiv_d a_i)\}$. This is a $\mathcal{L}$-definable set. Therefore by DC there $g\in G^+\cup\{\infty\}$ such that $g\in Y$ and either $g=\infty$ or $g+1\notin Y$. So we get that $X\cap [N,g)=\bigcup_i (a_i+dG)\cap [N,g)$ and $g$ is maximal for this property. As $X\cap [N,\infty)\cap \N=\bigcup_i (a_i+dG)\cap [N,\infty)\cap \N$, $g\notin \N$.

\end{proof}

We can now prove the generalisation of the result of Michaux-Villemaire.
\begin{theorem} Let $G\equiv \Z$ and $\mathcal{L}\supseteq \mathcal{L}_{Pres}$. Assume that $Th(G,\mathcal{L})$ admits EP and DC. Then for all $X\subset G$ definable. $X\cap \N$ is a finite union of points and of cosets of $d\N$ for some $d\in \N$.
\end{theorem}
\begin{proof}
	By Proposition \ref{EP implies no expanding set}, $Im\ \alpha$, $\widetilde{U}$, $Im\ D$ and $Im\ A$ are nonexpanding. So the result is an immediate consequence of Lemma \ref{loc definable}.
\end{proof}

This property is also true for $X\cap x+\N$ for all $x\in G^+$.

\begin{prop}\label{relative local definability} For all $x\in X$ there is $N(x)\in \N$ minimal, $d(x)\in \N_0$ and $g(x)\in G^+\cup\{\infty\}\setminus \N$ maximal such that $[x+N(x), x+g(x)]\cap X$ is a finite union of classes of $d(x)G$ intersected with $[x+N(x), x+g(x)]$.
\end{prop}
\begin{proof}By Lemma \ref{loc definable} applied with the set $X_x=\{y-x: y\in X\}$. For by Proposition \ref{EP implies no expanding set}, $Im\ \alpha$, $\widetilde{U}$, $Im\ D$ and $Im\ A$ are non expanding.
\end{proof}
\begin{Remark} We do not claim nor need that $N(x),d(x)$ or $g(x)$ are definable functions.
\end{Remark}

\begin{definition}
$$
\begin{array}{rl}
M:&G^+\times G^+ \rightarrow G^+\cup \{\infty,-1\}:\\
& (x,n)\longmapsto
\left\{\begin{array}{ll}
\max\{h\geq x+n:\ X[x,h-n]=X[x+n,h]\} & \mbox{if such $h$ exists in $G^+\cup \{\infty\}$;}\\
-1 & \mbox{ otherwise.}
\end{array}\right.
\end{array}
$$
\end{definition}

\begin{lemma}\label{bound on classes} Let $x\in X$ and $N=N(x),d=d(x),g=g(x)$ as given in Proposition \ref{relative local definability}. Then for all $x+N\leq y\leq x+N+\N$, if $n$ is finite and $d$ divides $n$ then $M(y,n)=x+g$. If $x\leq t<x+N(x)$ then $M(t,n)<x+N+n$.  Furthermore, for all $n$ finite, $M(x,n)\leq x+g$.
\end{lemma}
\begin{proof}
\par 1) First we prove that $M(y,n)=x+g$: By Proposition \ref{relative local definability} and definition of $X[x+N,x+g]$, $X[x+N,x+g]=\bigcup_{i<s}b_i+dG\cap [0,g-x]$ where $b_i$ is a representative for the classe of $a_i-x$. As $y\in x+N+\N$, there is $r\in \N$ maximal such that $x+N+rd\leq y$. So, 
$$\begin{array}{rll}
		[y,y+(d-1)]\cap X	&= & [x+N+rd+j,x+N+rd+(d-1)]\\
											&  & \bigcup[x+N+(r+1)d, x+N+(r+1)d+(j-1)]\cap X \mbox{ for some $j<d$}\\
											&= & \{a'_i+dr_i:\ i\leq s\}\\
											& & \mbox{ where  $r_i$ is either $r$ or $r+1$ and $a'_i\equiv_d b_i\equiv_d a_i$}
	\end{array}
$$
Therefore by the above description of $X[x,g]$, $[y,x+g]\cap X=(\{a'_i+dr_i:\ i<s\}+dG^+)\cap [y,x+g]$. So we remark that by definition of $X[y,x+g-n]$ (resp. $X[y+n, x+g]$)
$$X[y,x+g-n]=(X[y,y+(d-1)]+dG^+)\cap [0,x+g-n-y]\mbox{ and }$$ 
$$X[y+n,x+g]=(X[y+n,y+n+(d-1)]+dG^+)\cap [0,x+g-n-y].$$
Also by definition of $X[y,y+(d-1)]$, $t\in X[y,y+(d-1)]$ iff $y+t=a_i+dr_i$ for some $i$. This means that $X[y,y+(d-1)]=\{a'_i+dr_i-y,\ i<s\}$. Similarly $X[y+n,y+n+(d-1)]=\{a'_i+dr_i-y,\ i<s\}$.
This implies that $X[y,x+g-n]=X[y+n,x+g]$ i.e., $M(y,n)\geq x+g$. If $g=\infty$ we are done by definition of $M$. Otherwise assume that $M(y,n)>x+g$ i.e., assume that there is $g'>g$ $X[y,x+g'-n]=X[y+n,x+g']$.  By Lemma \ref{explicit X[a,b]}, $x+g+1-n\in X$ iff $x+g+1\in X$. So (as $d$ divides $n$) $[N+x,x+g+1]\cap X=\bigcup_{i<s}a_i+ dG\cap [N+x,x+g+1]$. This contradicts the maximality of $g$ in Proposition \ref{relative local definability}. So $M(y,n)=x+g$.
\par 2) $M(t,n)<x+N+n$: Assume that  $h:=M(t,n)\geq x+N+n$. By definition of $M(t,n)$, $X[t,h-n]=X[t+n,h]$. By Lemma \ref{increase} $X[x+N-1,h-n]=X[x+N(x)-1+n,h]$. So by Lemma \ref{explicit X[a,b]}, $x+N-1\in X$ iff $x+N-1+n\in X$. Then by Proposition \ref{relative local definability} $[x+N-1,x+g(x)]\cap X=\bigcup_i (a_i+dG)\cap[x+N-1,x+g(x)]$. This equality contradicts the minimality of $N$ in Proposition \ref{relative local definability}. 
\par 3) $M(x,n)\leq x+g$: Assume $h=M(x,n)>x+g$. By definition of $M(x,n)$, $X[x,h-n]=X[x+n,h]$. So by Lemma \ref{increase} $X[x+N,h-n]=X[x+N+n,h]$. By definition of $M(x+N,n)$ we get that $M(x+N,n)\geq h$. On the other hand by Lemma \ref{increase} again $X[x+N,x+g-n]=X[x+N+n,x+g]$. Then by Lemma \ref{explicit X[a,b]}, for all $x+N\leq y\leq x+g-n$, $y\in X$ iff $y+n\in X$. By Proposition \ref{relative local definability} this proves that $d$ divides $n$. Now as $M(x+N,n)\geq h>x+g$, we get a contradiction with step 1). So, $h=M(x,n)\leq x+g$.
\end{proof}

\begin{prop}\label{X Lpres-def}
$X$ is $\mathcal{L}_{Pres}$-definable.
\end{prop}
\begin{proof}
By Proposition \ref{relative local definability}, for all $x\in X$, there is $N(x), d(x)\in \N$ and $g(x)\in G^+\setminus \N$ maximal such that $X\cap [x+N(x),x+g(x)]$ is a union of classes of $d(x)G$ restricted to $[x+N(x),x+g(x)]$. So if there is $x_1,\cdots, x_k\in G^+$ such that $X\subseteq \bigcup_i [x_i,x+g(x_i)]$ we are done. Therefore for a contradiction we assume that for all $k\in \N$, for all $x_1,\cdots, x_k$ $X\not= \bigcup_i [x_i, x_i+g(x_i)]\cap X$.
\par For all $n\in\N$, there is $x_n$ such that for all $k\geq d(x_n)$, $M(x_n+N(x_n),k!)=x_n+g(x_n)$ and furthermore, if $n\not= n'$ there is $N$ such that for all $k\geq N$, $M(x_n+N(x_n),k!)\not=M(x_{n'}+N(x_{n'}),k!).$ Indeed, we can construct the $x_n$'s by induction: Take $x_0=0$. Then, by Lemma \ref{bound on classes} for all $k\geq d(0)$, $M(0+N(0),k!)=g(0)$. For all $n$, set $x_{n}:=x_{n-1}+g(x_{n-1})+1$. First remark that we may assume $x_n<\infty$. Otherwise, $x_{n-1}+g(x_{n-1})=\infty$. So, $G^+=\bigcup_{i<n}[x_n,x_n+g(x_n)]$. This contradicts our hypothesis. Now by Lemma \ref{bound on classes}, for all $k\geq d(x_n)$, $M(x_n+N(x_n),k!)=x_n+g(x_n)$. Also for all $k\geq \max\{d(x_i),\ i\leq n\}$,  $M(x_n+N(x_n),k!)=x_n+g(x_n)>x_n=x_{n-1}+g(x_{n-1})+1>x_i+g(x_i)=M(x_i+N(x_i),k!)$.
\begin{claim}For all $k\geq d(x_i)$, $M(x_i+N(x_i),k!)$ is not at finite distance from its predecessor in $Im\ M(\cdot, k!)$ (if any).
\end{claim}
Assume that the claim is true. Then we will build an elementary expansion of $G$ such that $Im\ M(\cdot, T)$ is an expanding set in this model (for some $T$). Then we get a contradiction with the exchange property (by Proposition \ref{EP implies no expanding set}) and the proposition is proved.
\par Let $G^*$ be an ultrapower of $G$ on a nonprincipal ultrafilter over $\N$. Let $T$ be the class of $(n!)_{(n\in \N)}$ and $y_i$ be the class of $(x_i+N(x_i))_{(n\in \N)}$ (the constant sequence). By construction of $x_i$ and \L os Theorem, if $i\not=j$, $M(y_i,T)\not= M(y_j,T)$. By the claim and \L os Theorem, for all $i,n\in \N$, $M(y_i,T)-n\notin Im\ M(\cdot, T)$ i.e., for all $i$, $M(y_i,T)$ is not at finite distance form its predecessor in $Im\ M(\cdot, T)$. This proves that $\{M(y_i,T),\ i\in \N\}$ is an infinite subset of $Im\ M(\cdot, T)$ and that each point in this set is not at finite distance from its predecessor in $Im\ M(\cdot, T)$. So, $Im\ M(\cdot, T)$ is an expanding set.

\par We give now a proof of the claim: By Lemma \ref{bound on classes} as $d(x_i)$ divides $k!$, $M(x_i+N(x_i),k!)=x_i+g(x_i)\not=-1$.  Let $t$ such that $M(t,k!)< M(x_i+N(x_i),k!)$. We have to prove that the distance between these two elements is not finite. There are three possible cases:
\par 1) if $x_i+N(x_i)<t$: in that case we may assume that $t< M(x_i+N(x_i), k!)-k!$. For if $t\geq M(x_i+N(x_i), k!)-k!$ then $M(t,k!)\geq t+k!\geq M(x_i+N(x_i), k!)+k!-k!$: contradiction with the choice of $t$. By definition of $M$, $X[x_i+N(x_i),M(x_i+N(x_i),k!)-k!]=X[x_i+N(x_i)+k!,M(x_i+N(x_i),k!)]$. So by Lemma \ref{increase}, $X[t,M(x_i+N(x_i),k!)-k!]=X[t+k!,M(x_i+N(x_i),k!)]$. This proves that $M(t,k!)\geq M(x_i+N(x_i), k!)$: contradiction. So case 1) never occurs.
\par 2) $x_i\leq t<x_i+N(x_i)$: Then by Lemma \ref{bound on classes}, $M(t,k!)<x_i+N(x_i)+k!$. As $M(x_i+N(x_i),k!)=x_i+g(x_i)$, $N(x_i),k!\in \N$ and $g(x_i)\notin \N$, we are done.
\par 3) $t<x_i$: Assume that $M(t,k!)$ is a finite distance from $M(x_i,k!)$. By Lemma \ref{bound on classes}, $M(t+N(t),k!)\leq t+g(t)$ where $N(t), g(t)$ are given by Proposition \ref{relative local definability}. Then $[t+N(t), t+g(t)]\cap [x_i+N(x_i), x+g(x_i)]$ is an infinite interval. Let $r=\max\{x_i+N(x_i), t+N(t)\}$, $s=\min\{x_i+g(x_i), t+g(t)\}$. By Proposition \ref{relative local definability},  $X\cap [r,s]=\bigcup_i a_i+d(x_i)G\cap [r,s]$ for some $0\leq a_i<d(x_i)$ and $X\cap [r,s]=\bigcup_i b_i+d(t)G\cap [r,s]$ for some $0\leq b_i<d(t)$. So we may replace $d(t)$ by $e=gcd(d(x_i),d(t))$. Also, $t+g(t)=x+g(x)$ (by maximality of $g(x),g(t)$). Now $d(t)$ divides $k!$. Then by Lemma \ref{bound on classes} $M(t+N(t),k!)=t+g(t)=x+g(x)=M(x_i+N(x_i),k!)$: contradiction.

\end{proof}

\bibliographystyle{plain}
\bibliography{Biblio_p-adic}

\noindent Nathana\"el Mariaule\\
Universit\'e de Mons, Belgium\\
\emph{E-mail address: Nathanael.MARIAULE@umons.ac.be}
\end{document}